\documentclass{article}
\usepackage{graphicx} % Required for inserting images
\usepackage{amsthm}
\title{Limit for Quotient Convergent Graph Sequence}
\date{March 2025}
\author{Yaobin Chen\thanks{Shanghai Center for Mathematical Sciences, Fudan University, Shanghai, China. Email: {\tt ybchen21@m.fudan.edu.cn}. Supported by National Natural Science Foundation of China (Grant No. 123B2012).}
\and
Zhicheng Liu\thanks{Institute of Operations Research and Information Engineering,  Beijing University of Technology, Beijing, China. E-mail: {\tt manlzhic@emails.bjut.edu.cn}. }
\and
Yihang Xiao\thanks{Shanghai Center for Mathematical Sciences, Fudan University, Shanghai, China. E-mail: {\tt yhxiao23@m.fudan.edu.cn}. }
\and
Junchi Zhang\thanks{Shanghai Center for Mathematical Sciences, Fudan University, Shanghai, China. E-mail: {\tt jczhang24@m.fudan.edu.cn}.}
}\date{\today}
\usepackage{geometry}
\usepackage[utf8]{inputenc}

\usepackage{amsmath} 
\usepackage{float}
\usepackage{verbatim}
\usepackage{enumitem}
\usepackage{mathrsfs}
\usepackage{amssymb}
\usepackage{csquotes}
\usepackage{amsfonts,amssymb}
\usepackage{booktabs}
\usepackage{diagbox}
\usepackage{tikz-cd}
\usepackage{geometry}
\usepackage{hyperref}
\usepackage{cleveref}
\usepackage{caption}
\usepackage{subcaption}
\newtheorem{thm}{Theorem}[section]

\newtheorem{ques}[thm]{Question}
\newtheorem{lem}[thm]{Lemma}

\newtheorem{prop}[thm]{Proposition}

\newtheorem{defi}[thm]{Definition}

\newtheorem{rem}[thm]{Remark}
\newtheorem{cor}[thm]{Corollary}

\newtheorem{clm}[thm]{Claim}
\usetikzlibrary{cd}
\usepackage{indentfirst}

\usepackage{setspace}
\everymath{\displaystyle}

\begin{document}

\maketitle
\begin{abstract}
Building on the limit theory for set functions introduced by~\cite{berczi2024quotient}, we prove that the limit of convergent sequence of bounded-degree graphs' cycle matroids can be represented as the cycle matroid of a graphing, analogous to the completeness result for local-global convergence in~\cite{MR3177383}.
\end{abstract}

\section{Introduction}

The \textit{cycle matroid} is a fundamental concept in both matroid theory and graph theory. It serves as a bridge between these two fields, providing a powerful framework for solving combinatorial and optimization problems.
In~\cite{thematroidofagraphing}, Lov{\'a}sz extended the notion of cycle matroids to graphing, which is defined as follows.

A \textbf{Borel graph} is a triple $(\Omega,\mathcal{B},E)$, where $(\Omega,\mathcal{B})$ is a \textit{standard Borel space} and $E$ is a symmetric Borel subset in $\Omega\times \Omega$.
A \textbf{graphing} is a bounded-degree Borel graph $\mathbf{G}=(\Omega,\mu ,E)$, where $(\Omega,\mathcal{B},E)$ is a Borel graph and $\mu$ is an probability measure on $(\Omega,\mathcal{B})$ that satisfies the \textit{involution invariant} property.
The rank function of a graphing $\mathbf{G}$ is defined as
$$\rho_{\mathbf{G}}(F)\triangleq1-\mathbb{E}_x[\frac{1}{|V(\mathbf{G}[F]_x)|}], \, \forall \,\text{ Borel edge subset } F,$$
where $x$ is a random point from $\mu$, $\mathbf{G}[F]_x$ denotes the connected component containing $x$ in the induced subgraphing  
$\mathbf{G}[F]=(\Omega,\mathcal{B},\mu ,F)$, and $V(G[F]_x)$ denotes the set of vertices of this component.

In~\cite{berczi2024quotient}, Krist{\'o}f B{\'e}rzi, M{\'a}rton Borb{\'e}nyi, L{\'a}szl{\'o} Lov{\'a}sz and L{\'a}szl{\'o} M{\'a}rton T{\'o}th developed a limit theory for the cycle matroids of graphings as below.
For a set function $\phi$ on a set-algebra $(\mathcal{J},\mathcal{B})$, 
a \textbf{$k$-quotient} of $\varphi$ is the function $\phi \circ F^{-1}$, where $F:J\rightarrow[k]$ is a measurable map. We denote by $\mathcal{Q}_k(\phi)$ the set of all $k$-quotients of $\phi$. A sequence of set functions $(\phi_1,\phi_2,\ldots )$ is said to \textbf{quotient converge} to $\phi$ if, for all $k\in \mathbb{N}$, $\mathcal{Q}_k(\phi_n)\rightarrow \mathcal{Q}_k(\phi)$ in the Hausdorff distance as $n\rightarrow \infty$.
In~\cite{berczi2024cycle}, it is shown by Krist{\'o}f B{\'e}rzi, M{\'a}rton Borb{\'e}nyi, L{\'a}szl{\'o} Lov{\'a}sz and L{\'a}szl{\'o} M{\'a}rton T{\'o}th that the local-global convergence (we do not explain the definition for brevity) of a sequence of finite graphs to a graphing implies convergence of the corresponding matroids. 
\begin{thm}\rm{([\cite{berczi2024cycle}, Theorem~3.2])}
Let $(G_i)_{i=1}^\infty$ be a sequence of finite graphs with maximal degree bounded by $D$ that converges to
a graphing $\mathbf{G}$ in the local-global sense. Then the rank function $\rho_i$ of $G_i$ quotient converges to $\rho_\mathbf{G}$.
\end{thm}

Moreover, in~\cite{MR3177383}, Hamed Hatami, L{\'a}szl{\'o} Lov{\'a}sz, and Bal{\'a}zs Szegedy prove the following completeness result.
\begin{thm}\rm{([\cite{MR3177383}, Theorem~3.2])}\label{lovaszcomplete}
Let $(G_i)_{i=1}^\infty$ be a local-global convergent sequence of finite graphs with maximal degree bounded by $D$. Then there exists a graphing $\mathbf{G}$ such that the $(G_i)_{i=1}^\infty$ is local-global convergent to $\mathbf{G}$.
\end{thm}

Hence, the following natural question is raised: Is the limit of cycle matroids of bounded-degree graph always a cycle matroid of some graphing? We prove it is true as follows.
\begin{thm}\label{quotientconvergentlimit}
Given $D\in \mathbb{N}^{*}$, let $\{G_i\}_{i=1}^\infty$ be a sequence of connected graphs with maximal degree bounded by $D$. Let $\rho_i$ be the rank function of the cycle matroid of $G_i$ as a graphing. If $\{\rho_i\}_{i=1}^\infty$ is quotient convergent, then there exists a graphing $\mathbf{G}$, such that the rank function $\rho_\mathbf{G}$ is a limit of $\rho_i$.
\end{thm}

\section{Preliminaries}\label{Preliminaries}

A \textbf{Borel graph} is a triple $\mathbf{G}=(\Omega,\mathcal{B},E)$ such that
\begin{enumerate}
        \item $E\subseteq \Omega\times \Omega$ is symmetric. That is, $(x,y)\in E$ if and only if $(y,x)\in E$.
    \item $(\Omega,\mathcal{B})$ is a standard Borel space, and $E$ is a Borel subset of $\Omega\times \Omega$.
\end{enumerate}

Here the Borel $\sigma$-algebra of $\Omega\times \Omega$ is given by $\sigma(\mathcal{B}\times \mathcal{B})$, the $\sigma$-algebra generated by $\mathcal{B}\times \mathcal{B}$.

In this paper, we require that $(x,x)\notin E$ for all $x\in \Omega$. Moreover, we can view $E$ as a quotient space of $\Omega\times \Omega$ such that $(x,y)\sim(y,x)$ in $E$. Then $E\subseteq \tbinom{\Omega}{2}$ and the topology on $E$ we consider is induced by the quotient. We use $e=[(x,y)]$ to denote the equivalence class containing $(x,y)$ in $E$. 

Given a measurable space $(X,\mathcal{B})$ and a $\sigma$-algebra $\mathcal{A}\subseteq 2^{X}$, for any Borel subset $Y\subseteq X$, we define the \textbf{restricted $\sigma$-algebra} of $\mathcal{B}$ on $Y$ as $\mathcal{A}|_Y \triangleq \{A\cap Y \mid A \in \mathcal{A}\}.$ 

\begin{lem}\label{polishspaceseperable}\rm{([\cite{pikhurko2008borel}, Lemma~3.4,~Corollary~3.11])}We give some basic properties of Polish space.
\begin{enumerate}[label=\rm{(T\arabic*)}]
    \item For any standard Borel space $(\Omega,\mathcal{B})$, there exists a countable family $\{J_n\in \mathcal{B} \mid n\in \mathbb{N}^{*}\}$, such that the Borel $\sigma$-algebra $\mathcal{B} = \sigma(\{J_n\in \mathcal{B} \mid n\in \mathbb{N}^{*}\})$, the $\sigma$-algebra generated by $\{J_n\in \mathcal{B} \mid n\in \mathbb{N}^{*}\}$.
    \item\label{substandard} If $(\Omega,\mathcal{B})$ is a standard Borel space and $Y \in \mathcal{B}$, then $(Y,\mathcal{B}|_Y)$
 is a standard Borel space.
\end{enumerate}
\end{lem}

\begin{lem}\label{inducedpolishspcae}We list some basic properties.
\begin{enumerate}[label={\rm{(P\arabic*)}}]
\item\label{inducedpolishspcae1} If $(\Omega, \mathcal{B})$ is a standard Borel space, then $(\Omega\times\Omega, \sigma(\mathcal{B}\times \mathcal{B}))$ and $\Omega\times \Omega/_\sim$ with the topology induced by the quotient such that $(x,y)\sim(y,x)$ in $\Omega$ are both standard Borel spaces.
\item\label{inducedpolishspcae4} 
For any Borel subset $X\subseteq \Omega$, $(X,\mathcal{B}|_X,E\cap (X\times X)/_\sim)$ is a Borel graph, called \textbf{the induced Borel graph of $(\Omega,\mathcal{B},\mu,E)$ on $X$}.
For any symmetric Borel subset $F\subseteq E$, $(\Omega,\mathcal{B},F)$ is also a Borel graph, called \textbf{the induced Borel graph of $(\Omega,\mathcal{B},\mu,E)$ on $F$}.
\item\label{inducedpolishspcae5} If we view $E\subseteq \tbinom{\Omega}{2}=\Omega\times \Omega/_\sim$ with the topology induced by the quotient and let $\mathcal{B}_E=\{X/_\sim\subseteq E\subseteq \tbinom{\Omega}{2} \mid X\in \sigma(\mathcal{B}\times \mathcal{B})$ and $X\text{ is symmetric}\}$, then $(E,\mathcal{B}_E)$ is a standard Borel space. Moreover, $\mathcal{B}_E$ is generated by a countable set. That is, there exists a countable sequence of Borel subsets of $E$, denoted by $\{F_1,F_2,\ldots,F_r,\ldots\}$, such that $\sigma(\{F_1,F_2,\ldots,F_r,\ldots\})=\mathcal{B}_E$.

\end{enumerate}
\end{lem}
\begin{proof}
 For \ref{inducedpolishspcae1}, the first part follows from the definition of standard Borel space. For $\Omega\times\Omega/_\sim$, we can assume the topology of $\Omega$ is induced by the metric $d_0$. Then we can define a function $d_1$ on $\Omega\times \Omega/_\sim$, such that $d_1([(x_1,x_2)],[(y_1,y_2)])=\min\{d_0(x_1,y_1)+d_0(x_2,y_2),d_0(x_1,y_2)+d_0(x_2,y_1)\}$.

We first check that $d_1$ is a metric. For $[(x_1,x_2)],[(y_1,y_2)],[(z_1,z_2)]\in \Omega\times \Omega/_\sim$, we can assume $$d_1([(x_1,x_2)],[(y_1,y_2)])=d_0(x_1,y_1)+d_0(x_2,y_2) \text{ and } d_1([(y_1,y_2)],[(z_1,z_2)])=d_0(y_1,z_1)+d_0(y_2,z_2).$$ Then $$d_1([(x_1,x_2)],[(y_1,y_2)])+d_1([(y_1,y_2)],[(z_1,z_2)])\geq d_0(x_1,z_1)+d_0(x_2,z_2)\geq d_1([(x_1,x_2)],[(z_1,z_2)]).$$ Moreover, if $d_1([(x_1,x_2)],[(y_1,y_2)])=0$, then either $(x_1,x_2)=(y_1,y_2)$ or $(x_1,x_2)=(y_2,y_1)$ and so $[(x_1,x_2)]=[(y_1,y_2)]$. Thus $d_1$ is a metric.

Next, it is easy to see that for $x=[(x_1,x_2)]$, $B_{r}(x)=\{y\in \Omega\times \Omega/_\sim\mid d_1(x,y)<r\}$ is the quotient of $B_{r}(\{(x_1,x_2),(x_2,x_1)\})$ in $\Omega\times \Omega$. Thus $B_{r}(x)$ is open in the induced topology. Moreover, for any open set $U_1$ in $\Omega\times \Omega/_\sim$, there exists an open set $U_0$ in $\Omega\times \Omega$, such that $U_1$ is the quotient of $U_0$. 
If we define a metric $d_2$ on $\Omega\times \Omega$ such that $d_2((x_1,x_2),(y_1,y_2))=d_0(x_1,y_1)+d_0(x_2,y_2)$, then the topology of $\Omega\times \Omega$ is induced by $d_2$, and there exists some $x_0\in U_0$ and $r_0>0$ such that $B_{r_0}(x_0)\subseteq U_0$. Let $x$ be the quotient of $x_0$ in $\Omega\times\Omega/_\sim$. Then it follows that $B_{r_0}(x)\subseteq U_1$. Thus the topology of $\Omega\times\Omega/_\sim$ induced by quotient is induced by metric $d_1$.

For the completeness, assume $\{[(x_i,y_i)]\}_{i\geq 1}$ is a Cauchy sequence under $d_1$. Then by taking a subsequence we can assume $d_1([x_i,y_i],[(x_j,y_j)])< 2^{-i-2}$ for any $1\leq i<j$. Then we construct a new sequence $\{(u_i,v_i)\}_{i\geq 1}$ in $\Omega\times\Omega$ such that $(u_1,v_1)=(x_1,y_1)$
and 
\begin{align*}
		(u_i,v_i)=\left\{ 
		\begin{aligned}
			(x_i,y_i),&\quad\quad \text{if }d_0((u_{i-1},v_{i-1}),(x_i,y_i))<2^{-i-1},\\
			(y_i,x_i),&\quad\quad \text{otherwise}.
		\end{aligned}  \right.
\end{align*}
Then $d_0((u_i,v_i),(u_{i+1},v_{i+1}))<2^{-i}$ and thus $\{(u_i,v_i)\}_{i\geq 1}$ is a Cauchy sequence in $\Omega\times\Omega$, and the limit is denoted by $w$. Let $z$ be the quotient of $w$ in $\Omega\times\Omega/_\sim$, then $w$ is the limit of $\{[(x_i,y_i)]\}_{i\geq 1}$. Thus $\Omega\times\Omega/_\sim$ is complete.

The separability of $\Omega\times\Omega/_\sim$ follows easily from the separability of $\Omega\times\Omega$. Thus $\Omega\times\Omega/_\sim$ with the quotient topology is also a standard Borel space. Hence~\ref{inducedpolishspcae1} holds. 

For~\ref{inducedpolishspcae4}, see Lemma 18.19 in \cite{Large_networks_and_graph_limits}, and \ref{inducedpolishspcae5} directly follows from~\Cref{polishspaceseperable} and~\ref{inducedpolishspcae4}.
\end{proof}

Throughout this paper, when we talk about $\mathcal{B}_E$, we always view $E$ as the quotient edge set. For a vertex $x\in \Omega$, the \textbf{degree of $x$} is the number of vertices $y\in \Omega$ such that $[(x,y)]\in E$, denoted by $d_\mathbf{G}(x)$.
Moreover, for any measurable set $A\subseteq X$, the \textbf{degree of $x$ in $A$} is the number of vertices $y\in A$ such that $[(x,y)]\in E$, denoted by $d_{\mathbf{G},A}(x)$.

In this paper, we only consider the Borel graphs such that \textbf{the degree of vertices in $\Omega$ is bounded by $D$, a fixed finite number}.

A measure $\mu$ on a standard Borel space $(\Omega,\mathcal{B})$ is called \textbf{involution invariant} if for all $A,B\in \mathcal{B}$:
\begin{align}\label{measurepreserving}
    \int_{A}\, d_{\mathbf{G},B}(x)\, d\mu(x)=\int_{B}\, d_{\mathbf{G},A}(x)\, d\mu(x).
\end{align}

Next, we give the formal definition of graphing.
\begin{defi}
    A \textbf{graphing} is a quadruple $\mathbf{G}=(\Omega,\mathcal{B},\mu,E)$, where $(\Omega,\mathcal{B},E)$
is a Borel graph with maximal degree bounded by a constant $D\geq 1$ and $\mu$ is an involution invariant probability measure on $(\Omega,\mathcal{B})$. A measure $\eta$ on $\Omega\times \Omega$ is induced by $\mu$ such that for any Borel subsets $A,B\subseteq \Omega$, $\eta(A\times B)=\dfrac{1}{2}\int_{A}\,d_B(x)\, d\mu(x)$. By Carath{\'e}odory's Theorem, there exists a unique measure on $\Omega\times \Omega$ induced by $\eta$, and this measure \textbf{induces a measure on $(E,\mathcal{B}_E)$}, denoted by $\tilde{\mu}$.
\end{defi}

For any measure on $X$, the restriction of $\mu$ on $Y$ is defined as $\mu|_Y$, such that for any $Y_1\in \mathcal{B}|_{Y}$, we have $\mu|_Y(Y_1)=\mu(Y)$.

\begin{defi}\label{inducegraphing}
    We can also define the induced graphing on vertex subsets and edge subsets. Given a graphing $(\Omega,\mathcal{B},\mu,E)$,
    \begin{enumerate}
        \item for any $X\in \mathcal{B}$ such that $\mu(X)>0$, $(X,\mathcal{B}|_X,\dfrac{\mu|_{X}}{\mu(X)},E\cap (X\times X))$ is a Borel graph, called \textbf{the induced graphing of $(\Omega,\mathcal{B},\mu,E)$ on $X$}.
        \item for any Borel set $F\in \mathcal{B}_E$, $(\Omega,\mathcal{B},\mu,F)$ is also a Borel graph, called \textbf{the induced Borel graph of $(\Omega,\mathcal{B},\mu,E)$ on $F$}.
    \end{enumerate}
\end{defi}

The \textbf{cycle matroid} of a graph $G=(V,E)$ is a pair $(E,\mathcal{I})$ consisting of the ground set $E$ and the set of independent sets $\mathcal{I}=\{F\in \mathcal{B}_E\mid \text{the induced graph } G[F]\text{ is acyclic}\}$. The rank of a subset of the ground set is the size of maximal independent set it contains.

For any edge set $F\subseteq E$, let $c(F)$ be the number of components in $(V,F)$. Then for the rank function $r_G$ of the cycle matroid of $G=(V,E)$, it follows that $$\dfrac{r_G}{|V(G)|}(F)=\dfrac{1}{|V(G)|}(|V(G)|-c(F))=\sum_{v\in V}\dfrac{1}{|V(G)|}(1-\frac{1}{|G[F]_{x}|}),\forall \,F\subseteq E,$$ where $G[F]$ is the graph $(V,F)$, and $|G[F]_x|$ is the size of the vertex set of the connected component of $G[F]$ that contains $x$.

For a graphing $\mathbf{G}=(\Omega, \mathcal{B},\mu,E)$, let $\mathcal{B}_E$ be the $\sigma$-algebra defined in~\Cref{inducedpolishspcae}. Two vertices $x,y\in\Omega$ are \textbf{connected} if there exists a finite path between $x,y$ in~$\mathbf{G}$. The \textbf{distance between two connected vertices} is the minimum length of paths connecting these two vertices. The rank function of the cycle matroid of the graphing $\mathbf{G}$, denoted by $\rho_\mathbf{G}(\cdot)$, is defined as:
$$\rho_{\mathbf{G}}(F)=\int_\Omega\, 1-\dfrac{1}{|\mathbf{G}[F]_x|}\, d\mu(x), \forall \,F\in \mathcal{B}_E. $$

For any finite connected graph $G=(V,E)$, we can regard it as a graphing $\mathbf{G}=(V,2^{[V]},\mu,E)$ with $\mu(S)=\frac{|S|}{|V|}$ for all $S\subseteq V(G)$. It is easy to show that $\mu$ is involution invariant.
Moreover, it follows that $|V(G)|\rho_{\mathbf{G}}$ is exactly the rank function of the cycle matroid of $G$ defined for finite graphs.

Given a graphing $\mathbf{G}=(\Omega,\mathcal{B},\mu,E)$, a \textbf{Borel $k$-coloring} of $\mathbf{G}$ is a Borel map $c_1\colon \Omega\longrightarrow [k]$.
A \textbf{Borel $k$-edge-coloring} of $\mathbf{G} $ is a Borel map $c_2: E(\mathbf{G}) \longrightarrow [k]$. We say $c_1$ is \textbf{proper} if for any edge $ij\in E(\mathbf{G})$, $c(i)\neq c(j)$; we say $c_2$ is \textbf{proper} if for any two incident edges $e_1$ and $e_2$, $c_2(e_1)\neq c_2(e_2)$.

\begin{prop}\label{Kechris}\rm{([\cite{KECHRIS19991}, Proposition~4.6])}
Let $\mathbf{G}$ be a graphing with vertex degree at most $D$. Then $\mathbf{G}$ has a proper Borel $(D+1)$-coloring and a proper Borel $(2D-1)$-edge-coloring.
\end{prop}

A measurable space is a pair $(\Omega,\mathcal{B})$ consisting of a set $\Omega$ and a $\sigma$-algebra $\mathcal{B}$ of subsets of $\Omega$.

\begin{defi}
    We give the definition of quotient convergence in~\cite{berczi2024quotient}. Given a measurable space $(\Omega,\mathcal{B})$, let $\varphi$ be a set function on $(\Omega,\mathcal{B})$.
    \begin{enumerate}
        \item A function $P\colon \Omega\longrightarrow \mathbb{N}$ is called \textbf{measurable} if $P^{-1}(k)\in\mathcal{B}$ for all $k\in\mathbb{N}$.
        \item For all $k\geq 1$, the \textbf{$k$-quotient set} $\mathcal{Q}_k(\varphi)$ is defined as
        $$\mathcal{Q}_k(\varphi)=\{\varphi\circ P^{-1}\mid P\colon \Omega\longrightarrow [k]\text{ is a Borel function}\}.$$
        \item For any $k\geq 1$ and $A,B\subseteq \mathbb{R}^{2^k}$, let $d(x,y)$ be the Euclidean distance between two points $x,y$. Then the Hausdorff distance between $A$ and $B$ is defined as
        $$d_{Haus}(A,B)\triangleq\max\{\sup_{x\in A} \inf_{y\in B} d(x,y), \sup_{x\in B} \inf_{y\in A} d(x,y) \}.$$
    \end{enumerate}
\end{defi}

Elements in $\mathcal{Q}_k(\varphi)$ can be embedded into $\mathbb{R}^{2^k}$ through the map $i(P)=(\varphi\circ P^{-1}(A))_{A\subseteq [k]}$. It follows that $d_{Haus}$ is a pseudometric on the space consisting of all $k$-quotient set of set functions on measurable spaces.

For two set functions $\varphi_1,\varphi_2$ on two 
Borel spaces $(\Omega_1,\mathcal{B}_1)$ and $(\Omega_2,\mathcal{B}_2)$ respectively, their distance is defined as:
$$d_Q(\varphi_1,\varphi_2)=\sum_{k=1}^\infty 2^{-k}d_{Haus}(\mathcal{Q}_k(\varphi_1),\mathcal{Q}_k(\varphi_2)).$$

This is still a pseudometric and gives a topology on the set of all set functions on measurable spaces. Note that this definition gives a way to compare two set functions on different measurable spaces.

\begin{defi}
    A sequence of set functions $\{\varphi_{n}\}_{n= 1}^\infty$ on some measurable spaces is said to \textbf{quotient converges} to some set function $\varphi$ on some measurable space if $\lim_{n\rightarrow\infty}d_Q(\varphi_n,\varphi)=0.$
    Equivalently, $\{\varphi_{n}\}_{n= 1}^\infty$ quotient converges to $\varphi$ if it converges to $\varphi$
    in the topology induced by the pseudometric $d_Q$.
\end{defi}

We now give some notion about the local profiles.

Given a graphing $\mathbf{G}=(\Omega, \mathcal{B},\mu,E)$,
the \emph{$r$-ball of a vertex $x\in\Omega$}, denoted by $B_{\mathbf{G},r}(x)$, is the induced subgraph on vertices with distance at most $r$ to $x$. We call a graphing (or graph) \emph{finite} or \emph{countable} if its vertex set is finite or countable respectively.
A \emph{directed graphing} is a graphing with an orientation on its edges.
A \emph{rooted graph} $(G,v)$ is a pair consisting of a graph $G$ and a vertex~$v\in V(G)$. 

We emphasize that when considering rank functions for directed graphs or graphings, we ignore the orientations on the edges.

For any two countable rooted (edge-colored directed) connected graphs $(G_1,v_1),(G_2,v_2)$, we say $(G_1,v_1)$ is isomorphic to $(G_2,v_2)$, if there exists an isomorphism $\varphi: V(G_1)\longrightarrow V(G_2)$, such that
\begin{enumerate}
    \item $\varphi(v_1)=v_2$.
    \item Any edge $xy\in E(G_1)$ if and only if $\varphi(x)\varphi(y)\in E(G_2)$. Thus $\varphi$ can also be viewed as a bijection from $E(G_1)$ to $E(G_2)$, such that $\varphi(xy)=\varphi(x)\varphi(y)$.
    \item If the graphs are edge-colored, then the color of $e$ is the same as the color of $\varphi(e)$, for all $e\in E(G_1)$.
    \item If the graphs are directed, then any directed edge $x\rightarrow y\in E(G_1)$ if and only if $\varphi(x)\rightarrow \varphi(y)\in E(G_2)$.
\end{enumerate}

For any edge-colored directed graphing $\mathbf{G}$ and any $k$-edge-coloring $\alpha:E(\mathbf{G})\rightarrow [k]$, the \textbf{quotient rank function} of $(\mathbf{G},\alpha)$ is defined as $\rho_{\mathbf{G}}\circ \alpha^{-1}:2^{[k]}\rightarrow \mathbb{R}$, denoted by $\rho_{\mathbf{G},\alpha}$, which can be considered as a point in $\mathbb{R}^{2^k}$. Then we can define the distance between two $k$-edge-colored directed graphings.

\begin{defi}\label{metric}
For any two $k$-edge-colored directed graphings $(\mathbf{G}_1,\alpha_1)$ and $ (\mathbf{G}_2,\alpha_2)$, the \textbf{$k$-quotient distance} between them is defined as $d_k((\mathbf{G}_1,\alpha_1), (\mathbf{G},\alpha_2))=d_{\mathbb{R}^{2^k}}(\rho_{\mathbf{G}_1,\alpha_1},\rho_{\mathbf{G}_2,\alpha_2})$, in which $d_{\mathbb{R}^{2^k}}$ is the Euclidean distance of $\mathbb{R}^{2^k}$. In particular, if $\mathbf{G}_1=\mathbf{G}_2=\mathbf{G}$, we define the \textbf{$k$-edge-coloring distance} between $\alpha_1$ and $\alpha_2$ to be $d_{k}(\alpha_1,\alpha_2)=d_{k}((\mathbf{G},\alpha_1), (\mathbf{G},\alpha_2))$. 
\end{defi}

\begin{rem}
    The distance $d_k$ may not be a metric, even if we restrict it on the set of $k$-edge-colorings of some fixed finite directed connected graph.
\end{rem}

To bridge the Borel graph $\mathbf{H}^K$ we construct in~Section~\ref{Proofofmain} with finite graphs, we need the following two lemmas.

\begin{lem}\label{densecoloring}
    For any $k,n\geq 1$, there exists $M(k,n)\in \mathbb{N}^{*}$, such that for any finite directed graph $G$, there exists a finite sequence of $k$-edge-colorings $A_{G,k,n}= (\alpha_1,\alpha_2,\ldots,\alpha_{M(k,n)})$, such that for any $k$-edge-coloring $\beta$ of $G$, $\underset{1\leq i\leq M(k,r)}{\min}d_k(\beta,\alpha_i)\leq 2^{-n}$. That is, $A_{G,k,n}$ is a finite $2^{-n}$-net of the set of $k$-edge-colorings of $G$. 
\end{lem}

\begin{proof}
    For any finite directed graph $G$ and $k$-edge-coloring $\alpha$, we have $0\leq \rho_{G,\alpha}\leq 1$. Thus $\rho_{G,\alpha}$ is in the compact set $E=\{(x_1,x_2,\cdots,x_{2^k}):0\leq x_i\leq 1, 1\leq i\leq 2^k\}\subseteq \mathbb{R}^{2^k}$. Let $\{\textbf{x}_1,\textbf{x}_2,\ldots,\textbf{x}_M\}$ be a $2^{-(n+1)}$-net of $E$. For any $1\leq i\leq M$, we select a $k$-edge-coloring $\alpha_i$ of $G$ such that $d_{\mathbb{R}^{2^k}}(\rho_{G,\alpha_i},\textbf{x}_i)\leq 2^{-(n+1)}$, if such a $k$-edge-coloring exists; otherwise, let $\alpha_i$ be an arbitrary $k$-edge-coloring of $G$. We claim that $(\alpha_1,\alpha_2,\cdots,\alpha_M)$ is a $2^{-n}$-net of the set of $k$-edge-colorings of $G$, (i.e., for every $k$-edge-coloring $\beta$ of $G$, there exists a $1\leq i\leq M$ such that $d_{\mathbb{R}^{2^k}}(\rho_{G,\beta}, \textbf{x}_i)\leq 2^{-(n+1)}$). By the definition of $\alpha_i$, we have \[d_k(\beta,\alpha_i)=d_{\mathbb{R}^{2^k}}(\rho_{G,\beta},\rho_{G,\alpha_i})\leq d_{\mathbb{R}^{2^k}}(\rho_{G,\beta},\textbf{x}_i)+d_{\mathbb{R}^{2^k}}(\rho_{G,\alpha_i},\textbf{x}_i)\leq 2^{-(n+1)}+2^{-(n+1)}=2^{-n}.\]
\end{proof}

\begin{rem}
Note that in~\Cref{densecoloring} we do not need to assume that the maximal degree of $G$ is bounded by a constant $D$. This is different from Lemma~19.18 proved by Lov{\'a}sz in \cite{Large_networks_and_graph_limits}.
\end{rem}

\begin{lem}\label{continuouscoloring}\rm{([\cite{Large_networks_and_graph_limits}, Lemma~19.17])}
    Let $K$ be a compact and totally disconnected (i.e. any two points are not connected) metric space with a probability measure $\pi$ on $K$. For any   Borel functions $\alpha: K\rightarrow [k]$ and any $\varepsilon >0$, there exists a continuous function $\delta: K\rightarrow [k]$ such that $\pi(\{\alpha\neq \delta\})\leq \varepsilon$.
\end{lem}
The following Prokhorov's theorem plays an important role in finding the limit object.
\begin{thm}\label{Prokhorov}\rm{([\cite{billingsley2013convergence}, Theorem~5.2])}    If $(X,d)$ is a compact metric space, then $\mathscr{P}(X)$ is a compact metric space with respect to $d_{P}$. That is, for any sequence of Borel probability measures $\{\mu_i\}_{i=1}^\infty$ on $X$, there exists a convergent subsequence $\{\mu_{i_k}\}_{k=1}^\infty$ with respect to $d_{P}$. In particular, for any bounded continuous function $f:X\longrightarrow \mathbb{R}$, we have
$$\lim_{k\rightarrow \infty}\int_X\,f\, d\mu_{i_k}=\int_X\,f\, d\mu.$$
\end{thm}

\section{Limit for Quotient Convergent Graph Sequences}\label{Proofofmain}
In this section, we aim to prove~\Cref{quotientconvergentlimit}.
We construct an edge colored directed Borel graph $\mathbf{H}^K$ following the method in \cite{MR3177383} that is used to prove~\Cref{lovaszcomplete}. Briefly, given a sequence of convergent graphs, the authors construct a space consisting of finite vertex-labelled connected graphs, regard these graphs as vertices of a Borel graph and define the edges between these vertices in a particular way, such that every finite graph can be embed into the Borel graph. Moreover, for any convergent graph sequence, they define a sequence of involution measures on the Borel graph to make it isomorphic to the graphs as a graphing respectively, and then use Prokhorov's Theorem (\Cref{Prokhorov}) to find a weak limit and obtain the required graphing.

In our proof, given a graph sequence whose rank functions are quotient convergent, since what we consider is the rank function of the cycle matroid, we focus on the edges of the graphs rather than the vertices. So we consider the space $\Omega$ consisting of all countable rooted directed connected graphs with edge colored by some decoration space $K$. Here, we carefully choose an injective edge coloring of every finite directed graph to distinguish edges and add an orientation on the edges to distinguish the two vertices of each edge. Thus we can embed any finite directed graph to this Borel graph $\mathbf{H}^K$ (as in~\Cref{injectembe}). For any such graphs, we can give an involution invariant measure on $\Omega$ to obtain a graphing isomorphic to this graph as a graphing. We can prove that $\Omega$ is a compact space, and thus by Prokhorov's Theorem, the space of probability measures on $\Omega$ is compact. So we can choose a weak limit measure $\mu$ of some subsequence of the corresponding measures of the graphs in the graph sequence in~\Cref{quotientconvergentlimit}.
Then we get a graphing and try to prove that the rank function of this graphing is exactly the limit of the rank functions of the graph sequence.\\

First, we introduce some notations. 
In~\Cref{densecoloring}, for any $k,n\geq 1$, we can fix a number $M(k,n)$, such that for any finite directed graph $G$, there exists a $2^{-n}$-net of size $M(k,n)$ in the set of all $k$-edge-colorings of $G$.
Let $K=\prod_{k,n=1}^\infty [k]^{M(k,n)}$ be the \textbf{decoration space}
and $P_m: K\longrightarrow \prod_{k,n=1}^m [k]^{M(k,n)}$ be the natural projection. Let $\Omega_0$ be the set of all triples $(G,v,\chi)$, where $(G,v)$ is a rooted directed countable connected graph with maximal degree bounded by~$D$, and $\chi$ is a map from the $E(G)$ to $K$, called a \textbf{decoration} of $G$.
Two points $(G_1,v_1,\chi_1),(G_2,v_2,\chi_2)$ are equivalent if there exists an isomorphism between $G_1,G_2$ that preserves the root and the decoration. That is, there exists an bijection $\phi$ from $V(G_1)$ to $V(G_2)$, such that
\begin{enumerate}[label=(\roman*)]
    \item $\phi(v_1)=v_2$.
    \item Any directed edge $x\rightarrow y\in E(G_1)$ if and only if $\phi(x)\rightarrow \phi(y)\in E(G_2)$. Thus $\phi$ can also be viewed as a bijection from $E(G_1)$ to $E(G_2)$), such that $\phi(x\rightarrow y)=\phi(x)\rightarrow\phi(y)$.
    \item $\chi_1(e)=\chi_2(\phi(e))$, for all $e\in E(G_1)$.
\end{enumerate}
This definition remains valid if $\chi$ is a map from $E(K)$ to $P_m(K)$, allowing us to define equivalence relation for triples like $(G,v,P_m(\chi))$.
Let $\Omega$ be the quotient space under this equivalence relation.

For any rooted directed graph $G$, let $B_{G,n}(v)$ be the induced graph of $G$ on the set of vertices with distance at most $n$ to $v$.
If $\chi$ is a decoration of $G$, then the restriction of this function on $E(B_{G,n}(v))$ is a decoration of $B_{G,n}(v)$, denoted by $\chi |_{B_{G,n}(v)}$.
For any two triples $(G_1,v_1,\chi_1)$ and $(G_2,v_2,\chi_2)$ in $\Omega$, the \textbf{local distance} on $\Omega$ is defined as 
\begin{align*}
     d_{\Omega}((G_1,v_1,\chi_1),(G_2,v_2,\chi_2)) =\inf_{n\geq 0}\left\{2^{-n}+2^{-m} \;\middle|\; \left. 
		\begin{matrix}
			(B_{G_1,n}(v_1),v_1,P_m\circ\chi_1|_{B_{G_1,n}(v_1)})\\
			=(B_{G_2,n}(v_2),v_2,P_m\circ\chi_2|_{B_{G_2,n}(v_2)})
		\end{matrix}  \right.    
        \right\}.
\end{align*}

It can be verified that this distance defines a metric on $\Omega$. We call the topology induced by this metric the \textbf{local topology}. Moreover, let the local topology of $\Omega\times \Omega$ be given by the product topology.

\begin{lem}\label{completivity}
Let $\Omega$ and $d_\Omega$ be defined as above. Then $(\Omega,d_{\Omega})$ is a complete metric space.
\end{lem}

\begin{proof}
     Suppose $\{(G_n,v_n,\chi_n)\}_{n=1}^\infty$ is a Cauchy sequence with respect to $d_{\Omega}$. Then for any $s\in \mathbb{N}^*$, there exist a constant $N(s)$ such that for every $n,m\geq N(s)$, $$d_{\Omega}((G_n,v_n,\chi_n),(G_m,v_m,\chi_m))< 2^{-s}.$$
     By taking a subsequence, we may assume that $N(s)=s$.
     We construct the limit point in~$\Omega$ by induction on~$n$.

     Let $T_1=(B_{G_1,1}(v_1),v_1,\chi_1|_{B_{G_1,1}(v_1)})$. If we have defined $T_i=(H_i,v_1,c_i)$ for $i\leq n$, such that $d(T_n,G_m)<2^{-n}$ for any $m\geq n$, and $B_{H_n,m}(v_1)=H_m$ and $c_n|_{E(H_m)}=c_m$ for any $1\leq m<n$.
     Then there exists an isomorphism $\phi: B_{G_{n+1},n}(v_{n+1})\rightarrow H_n$ such that $$\phi(v_{n+1})=v_1,~~~ P_n\circ\chi_{n+1}(e)=P_n\circ c_n(\phi(e))~\text{for all }e\in E(B_{G_{n+1},n}(v_{n+1})).$$ Next, we can extend $H_n$ to a graph $H_{n+1}$ and extend $c_n$ to a decoration $c_{n+1}$ of $H_{n+1}$, such that there exists an isomorphism $\phi': B_{n+1,G_{n+1}}(v_{n+1})\rightarrow H_{n+1}$ with $$\phi'(v_{n+1})=v_1,~~~  P_{n+1}\circ \chi_{n+1}(e)=P_{n+1}\circ c_{n+1}(\phi'(e)), ~\text{for all } e\in E(B_{G_{n+1},n+1}(v_{n+1})).$$  Let $T_{n+1}=(H_{n+1},v_1,c_{n+1})$. Note that the limit of $T_n$ is a triple $T=(H,v_1,\chi)$, such that $H$ is a countable directed connected graph rooted at $v_1$, $\chi$ is a decoration of $H$. Then $T$ is the limit of $\{(G_n,v_n,\chi_n)\}_{n=1}^\infty$ in $\Omega$.
\end{proof}

\begin{lem}\label{Omegacompa}
Let $\Omega$ be defined as above. Then under the local topology, $\Omega$ is compact and totally disconnected (i.e. any two points are not connected).
\end{lem}

\begin{proof}
    The total disconnectedness follows from the definition.
    By~\Cref{completivity}, it suffices to show that $\Omega$ is sequentially compact. Let $\{(G_n,v_n,\chi_n)\}_{n=1}^\infty$ be an arbitrary sequence in $\Omega$. For any $r\geq 1$, consider the triples $\{(B_{G_n,r}(v_n),v_n,P_r\circ \chi_n|_{E(B_{G_n,r})}\}_{n=1}^\infty$. Since there exist finite non-isomorphic choices of such triples, repeatedly by pigeonhole principle, we can select subsequences $\mathcal{G}_1\supseteq\ldots \supseteq\mathcal{G}_r\supseteq \ldots $ of  $\{(G_n,v_n,\chi_n)\}_{n=1}^\infty$ such that $$\text{for all $(G_n,v_n,\chi_n)$ in }\mathcal{G}_r,~~~\{(B_{G_n,r}(v_n),v_n,P_r\circ \chi_n|_{E(B_{G_n,r})}\}_{n=1}^\infty\text{ are the same}.$$ 

    So in $\mathcal{G}_r$, any two elements have distance smaller than $2^{-r}$. We construct a new subsequence $\mathcal{G}'=\{G_n'\}$ such that $G_n'\in\mathcal{G}_n$ for $n\ge 1$. Then $\mathcal{G}'$ is a convergent sequence under the metric~$d_\Omega$. Hence, $\Omega$ is compact. 
\end{proof}
Let $$\mathcal{D}=\bigcup_{r\geq1} \{(G,v,\chi)\in \Omega\mid |V(G)|=r, v\in V(G),\chi(e)\in \prod_{k,n=1}^r [k]^{M(k,n)}\cdot\prod_{k>r\text{ or }n>r}[1]^{M(k,n)}\}.$$ Then $\mathcal{D}$ is a countable set dense in $\Omega$. Combined with~\Cref{inducedpolishspcae},~\Cref{completivity} and~\Cref{Omegacompa}, we obtain the following corollary.
\begin{cor}
Let $\Omega$ be defined as above, then $\Omega$ and $\Omega\times\Omega$ are both compact and totally disconnected standard Borel spaces when equipped with the local topology.
\end{cor}

We denote the Polish space $\Omega$ with the local topology by $(\Omega,\mathcal{B})$, where $\mathcal{B}$ is the Borel $\sigma$-algebra of $\Omega$.
Now we can define an directed edge set $E$ on $\Omega$ and a directed Borel graph $\mathbf{H}^K=(\Omega,\mathcal{B},E)$.

For two vertices $(G_1,v_1,\chi_1),(G_2,v_2,\chi_2)$, a directed edge from $(G_1,v_1,\chi_1)$ to $(G_2,v_2,\chi_2)$ belongs to $E$ if and only if there exists a countable directed connected graph $H$ with a directed edge $u\rightarrow u'$ and a decoration $\chi$ of $H$, such that $(H,u,\chi)=(G_1,v_1,\chi_1)$ and $(H,u',\chi)=(G_2,v_2,\chi_2)$ in $\Omega$.

\begin{lem}\label{EHKcop}
    Let $E(\mathbf{H}^K)$ be defined as above. Then
    $E(\mathbf{H}^K)$ is closed in the local topology of $\Omega\times \Omega$ and so $E(\mathbf{H}^K)$ is compact and totally disconnected. In particular, $E(\mathbf{H}^K)$ is Borel and $(\Omega,\mathcal{B},E(\mathbf{H}^K))$ is a Borel graph.
\end{lem}

\begin{proof}
    It suffices to prove that $(\Omega\times\Omega)\setminus E(\mathbf{H}^K)$ is open. For two vertices $(G_1,v_1,\chi_1)$ and $(G_2,v_2,\chi_2)$, if there does not exist edge between them, by the definition of $E(\mathbf{H}^K)$, there exist $r\geq 1$, such that there exists 
    no countable directed connected graph $H$ with a directed edge $u\rightarrow u'$ and a decoration $\chi$ of $H$, such that $$(B_{H,r}(u),u,P_r\circ\chi|_{E(B_{H,r})})=(B_{G_1,r}(v_1),v_1,P_r\circ\chi_1|_{E(B_{G_1,r}(v_1))})$$ and $$(B_{H,r}(u'),u',P_r\circ\chi|_{E(B_{H,r}(u'))})=(B_{G_2,r}(v_2),v_2,P_r\circ\chi_2|_{E(B_{G_2,r}(v_2))}),$$ otherwise as in~\Cref{completivity} we can construct a limit graph to guarantee that $u\rightarrow u'$ is an edge in $\mathbf{H}^K$.
    
    Now we claim that if $$d_{\Omega}((G_1,v_1,\chi_1),(H_1,w_1,c_1))<2^{-r}\text{ and }d_{\Omega}((G_2,v_2,\chi_2),(H_2,w_2,c_2)))<2^{-r},$$ then there does not exist edge between $(H_1,w_1,c_1)$ and $(H_2,w_2,c_2)$ and so $\Omega\times \Omega\setminus  E(\mathbf{H}^K)$ is open. Otherwise, suppose that there exists a directed edge from $(H_1,w_1,c_1)$ to $(H_2,w_2,c_2)$. By the definition of $E(\mathbf{H}^K)$, there exists a countable directed connected graph $H_3$ with a directed edge $w\rightarrow w'$ and a decoration $c$ such that $(H_3,w,c)=(H_1,w_1,c_1)$ and $(H_3,w',c)=(H_2,w_2,c_2)$.
    By the definition of local distance, it follows that $$(B_{G_1,r}(v_1),v_1,P_r\circ \chi_1|_{E(B_{G_1,r}(v_1))})=(B_{H_1,r}(w_1),w_1,P_r\circ c_1|_{E(B_{H_1,r}(w_1))})=(B_{H_3,r}(w),w,P_r\circ c|_{E(B_{H_3,r}(w))}),$$ $$(B_{G_2,r}(v_2),v_2,P_r\circ \chi_2|_{E(B_{G_2,r}(v_2))})=(B_{H_2,r}(w_2),w_2,P_r\circ c_2|_{E(B_{H_2,r}(w_2))})=(B_{H_3,r}(w'),w',P_r\circ c|_{E(B_{H_3,r}(w'))}),$$
    a contradiction.
\end{proof}

The isomorphism between finite directed graphs (i.e., bijection between the vertex sets that preserves edges and directions) is an equivalence relation and for a directed graph $G$, we denote $[G]$ to be the equivalence class containing $G$. For each equivalence class $[G]$ we select a fixed representative element $G_0\in [G]$, and for any $H$ in $[G]$ we select a fixed isomorphism $\varphi_H$ from $H$ to $G_0$ (in particular, let $\varphi_{G_0}=id_{G_0}$, the identity map). For any graph isomorphism $\varphi:G_1\longrightarrow G_2$, we can simultaneously view it as a bijection from $V(G_1)$ to $ V(G_2)$ and a bijection from $E(G_1)$ to $E(G_2)$.
Recall that we have fixed $M(k,n)$ defined in~\Cref{densecoloring}, for $k,n\geq 1$ and for any representative element $G$, we can fix a sequence $A_{G,k,n}$, which is a $2^{-n}$-net of size $M(k,n)$ in the set of all $k$-edge-colorings of $G$, by~\Cref{densecoloring}. For any $H$ in $[G]$, let $A_{H,k,n}= (\alpha_1\circ \varphi_H,\alpha_2\circ\varphi_H,\ldots,\alpha_{M(k,n)}\circ\varphi_H)$. Then $A_{H,k,n}$ is a $2^{-n}$-net of the set of $k$-edge-colorings of $H$ with size $M(k,n)$ for any finite directed graph $H$.

For any finite directed graph $G$, define $\chi_{G}:E(G)\rightarrow K$, $\chi_{G}(e)=(\alpha_i(e))_{\alpha_i\in A_{G,k,n}}$.
Then for any two triples $(G_1,v_1,\chi_{G_1}), (G_2,v_2,\chi_{G_2})\in\Omega$, $(G_1,v_1,\chi_{G_1})=(G_2,v_2,\chi_{G_2})$ if and only if $[G_1]=[G_2]$ and $\phi_{G_1}(v_1)=\phi_{G_2}(v_2)$.

\begin{lem}\label{injectembe}
    Given a finite directed graph $G$, the map $\chi_G$ is injective. Let $\tau_G$ be the map from $V(G)$ to $\Omega$ such that $\tau_G(v)= (G,v,\chi_G)$. Then $\tau_G$ is also injective, and for any two vertices $x_1,x_2\in V(G)$, the directed edge $x_1\rightarrow x_2\in E(G)$ if and only if $\tau_G(x_1)\rightarrow \tau_G(x_2)\in E(\mathbf{H}^K)$. Equivalently, $\tau_G$ is a graph isomorphism from $G$ to the induced graph of $\mathbf{H}^K$ on $\tau_G(V(G))$.
\end{lem}
\begin{proof}
    We first prove that $\chi_G$ is injective. Otherwise we assume that $\chi_G(e_1)=\chi_G(e_2)$ for two different edges $e_1,e_2\in E(G)$. Let $c$ be an $E(G)$-edge-coloring on $E(G)$, such that $c(e_3)\neq c(e_4)$ for any two different edges $e_3,e_4\in E(G)$. Then $\rho_G\circ c^{-1}(\{i\})=\dfrac{1}{|V(G)|}$, for any $i\in [k]$.    
    By the definition of $\chi_G$, there exists a sequence of $E(G)$-edge-colorings $\{c_n\}$, such that $d_{|E(G)|}(c_n,c)<2^{-n}$, and $c_n\in A_{G,k,n}$. Since $\chi_G(e_1)=\chi_G(e_2)$, it follows that $c_n(e_1)=c_n(e_2)$. Then $\rho_G\circ c_n^{-1}(c_n(e_1))\geq \dfrac{2}{|V(G)|}$, and thus $d_{|E(G)|}(c,c_n)\geq \dfrac{1}{|V(G)|}$, a contradiction. 

    Now we prove that $\tau_G$ is injective. Otherwise we assume that
    $\tau_G(v_1)=\tau_G(v_2)$ for two vertices $v_1,v_2\in V(G)$. By the definition of $\Omega$, it follows that $(B_{G,1}(v_1),v_1,\chi_G)=(B_{G,1}(v_2),v_2,\chi_G)$. Thus we have 
    $$\bigcup_{v_1\in e\in E(G)}\{\chi_G(e)\}=\bigcup_{v_2\in e \in E(G) }\{\chi_G(e)\}.$$
    Note that $\chi_G$ is injective, so there exists an edge $e\in E(G)$, such that $v_1,v_2\in e$, and an isomorphism $\varphi$ between $B_{G,1}(v_1)$ to $B_{G,2}(v_2)$, such that $\varphi(e)=e$. However, $e$ is an directed edge, a contradiction.

    The last proposition follows from the definition of $E(\mathbf{H}^K)$.
\end{proof}

By~\Cref{injectembe}, for any finite directed graph $G$, we can define a corresponding probability measure $\mu_G$ on $\Omega$, such that $\mu_G(X)=|X\cap \{(G,v,\chi_G)\in \Omega\mid v\in V(G)\}|$, for any $X\in \mathcal{B}_E$, which is obviously involution invariant, and the graphing
$(\Omega,\mathcal{B},\mu_G,E(\mathbf{H}^K))$ is equal to $G$ as a graphing if we neglect a null set, and so their local profiles are the same. In particular, they have the common rank function under the isomorphism (up to a null set) $\tau_G$ as defined in~\Cref{injectembe}.\\

Now we are ready to prove~\Cref{quotientconvergentlimit}.

\begin{proof}[Proof of~\Cref{quotientconvergentlimit}]
    Suppose that $\{G_i\}_{i=1}^\infty$ is a sequence of finite graphs with maximal degree bounded by $D$. We give an arbitrary orientation on edges of $\{G_i\}_{i=1}^\infty$ and turn it into a sequence of finite directed graphs.
By~\Cref{Prokhorov},~\Cref{completivity} and~\Cref{Omegacompa}, the set of probability measures on $\Omega$ is compact with respect to weak topology.
So the sequence of measures $\{\mu_{G_i}\}_{i\ge 1}$ has a subsequence $\{\mu_{G_{i_k}}\}_{k\ge 1 }$ weakly convergent to a measure $\mu$ on $\Omega$.
For any two open sets $A,B\subseteq \Omega$, let $$f(x)=|\{(x,y)\in (A\times B)\cap E\}| \text{ and } g(x)=|\{(x,y)\in (B\times A)\cap E\}|.$$ Note that $f,g$ are bounded and continuous. Since $\mu_{G_{i_k}}$ is involution invariant for any $k\geq 1$, it follows that
\begin{align*}
    \int_{A} d_{B}(x)d\mu(x)=\int_{\Omega}fd\mu= \lim_{k\rightarrow\infty}\int_{\Omega}f d\mu_{G_{i_k}}=\lim_{k\rightarrow\infty}\int_{\Omega}gd\mu_{G_{i_k}}= \int_{\Omega}gd\mu=\int_{B} d_A(x)d\mu(x).
\end{align*}
Moreover, for any 
Borel sets $A,B$ and disjoint unions $\bigcup_{\geq 1}A_i$ and $\bigcup_{\geq 1}B_i$, we have $$\int_{A^c} d_{B}(x)d\mu(x)=\int_\Omega d_B(x)d\mu(x)-\int_{A} d_B(x)d\mu(x),$$ $$\int_{\bigcup_{i\geq 1} A_i}d_{B}(x)d\mu(x)=\sum_{i\geq 1}\int_{A_i} d_B(x)d\mu(x),$$ $$\int_{A} d_{B^c}(x)d\mu(x)=\int_\Omega (d_\Omega(x)-d_B(x))d\mu(x),$$ $$\int_{A}d_{\bigcup_{i\geq 1} B_i}(x)d\mu(x)=\sum_{i\geq 1}\int_{A} d_{B_i}(x)d\mu(x).$$  Note that $\mathcal{B}$, the $\sigma$-algebra of $\Omega$, is generated by open sets of $\Omega$. It turns out that for any Borel sets $A,B\in \mathcal{B}$, we have $\int_{A} d_{B}(x)d\mu(x)=\int_{B} d_{A}(x)d\mu(x)$. So $\mathbf{G}=(\Omega,\mathcal{B},\mu,E(\mathbf{H}^k))$ is a directed graphing.

\begin{clm}\label{aeinjective}
    Let $A=\{(G,v,\chi)\in \Omega\mid \chi$ is not injective$\}$. Then $\mu(A)=0$.
\end{clm}
\begin{proof}
    Let $A_n=\{(G,v,\chi)\in \Omega\mid $ there exist $e_1,e_2\in B_{G,n}(v)$, such that $\chi(e_1)=\chi(e_2)\}$.
    Since $\mu$ is countably additive, it suffices to show that $\mu(A_n)=0$ for all $n\geq 1$.
    By definition, $\mu(A_n)=\int_{\Omega}1_{\{x\in A_n\}}d\mu(x)$. Note that the function $1_{\{x\in A_n\}}$ is bounded and continuous. Meanwhile for $i\ge 1$, $\int_{\Omega}1_{\{x\in A_n\}}d\mu_{G_i}(x)=0$ since $\chi_{G_i}$ are injective.
    Thus, $$\mu(A_n)=\int_{\Omega}1_{\{x\in A_n\}}d\mu(x)=\lim_{i\rightarrow\infty }\int_{\Omega}1_{\{x\in A_n\}}d\mu_{G_i}(x)=0.$$
\end{proof}

\begin{clm}
    $\mathcal{Q}_k(\mathbf{G})\subseteq \overline{\underset{i\rightarrow\infty}{\lim} \mathcal{Q}_k(G_i)}$.
\end{clm}

\begin{proof}
    By~\Cref{continuouscoloring} and~\Cref{EHKcop}, for any Borel $k$-edge-coloring $c:E(\mathbf{H}^K)\longrightarrow [k]$ and  $\epsilon>0$, we can find a continuous $k$-edge-coloring $c_{\epsilon}$ such that $\tilde{\mu}(c\neq c_\epsilon)<\epsilon.$ Since $\rho_\mathbf{G}\leq \tilde{\mu}$, this implies that $d_{k}(c,c_\epsilon)<2^{k/2}\epsilon$. Let $\tau_{G_i}$ be the the embedding map of $V(G_i)$ into $\Omega$ and $c_{i,\epsilon}=c_\epsilon\circ \tau_{G_i}$. Then~$c_{i,\epsilon}$ can be viewed as a coloring of $G_i$ and for any subset $F\subseteq [k]$,
    \begin{align*}
        \lim_{k\rightarrow\infty}\rho_{G_{i_k}}\circ c_{i_k,\epsilon}^{-1}(F)&=\lim_{k\rightarrow\infty}\sum_{v\in V(G)}\frac{1}{|V(G)|}(1-\dfrac{1}{|G[c_{i_k,\epsilon}^{-1}(F)]_x|})=\lim_{k\rightarrow\infty}\int_{\Omega} 1-\dfrac{1}{|\mathbf{G}[c_\epsilon^{-1}(F)]_x|}  d\mu_{G_{i_k}}\\
        &=\int_{\Omega} 1-\dfrac{1}{|\mathbf{G}[c_\epsilon^{-1}(F)]_x|}  d\mu_G=\rho_{\mathbf{G}}\circ c_\epsilon^{-1}(F).
    \end{align*}

Note that $\lim_{\epsilon\rightarrow0}\rho_\mathbf{G}\circ c^{-1}_\epsilon(F)=\rho_\mathbf{G}\circ c^{-1}(F)$ for any $F\subseteq [k]$. Thus 
$\rho_\mathbf{G}\circ c^{-1} \in \overline{\underset{i\rightarrow\infty}{\lim} \mathcal{Q}_k(G_i)}$.
\end{proof}

\begin{clm}
    $\overline{\mathcal{Q}_k(\mathbf{G})}\supseteq \underset{i\rightarrow\infty}{\lim} \mathcal{Q}_k(G_i)$
\end{clm}
\begin{proof}
For any sequence of $k$-edge-colorings $c_i: V(G_i)\longrightarrow [k]$ such that $\rho_{G_i}\circ c_i^{-1}(A)$ converges for all $A\subseteq [k]$ and any $r\geq 1$, by~\Cref{densecoloring} we can find $n_{i,r}\in [M(k,r)]$ and a coloring $c_{i,r}$ such that
$d(c_i,c_{i,r})<2^{-r}$, and $c_{i,r}$ is the $n_{i,r}$-th coloring of the sequence $A_{G_{i},k,r}$. Then there exists a subsequence $c_{i_j}$ such that $n_{i_j,r}$ is equal to some $n_r\in [M(k,r)]$ for all $j\geq1$ and $r\geq 1$.

 If there exists a directed edge $e$ from $(G_1,v_1,\chi_1)$ to $(G_2,v_2,\chi_2)$, then there exists a countable directed connected graph $H$ with a directed edge $u\rightarrow u'$ and a decoration $\chi_0$ of $H$, such that $(H,u,\chi_0)=(G_1,v_1,\chi_1)$ and $(H,u',\chi_0)=(G_2,v_2,\chi_2)$ in $\Omega$. Let $\chi_\Omega: E(\mathbf{H}^K)\longrightarrow K$ be such that $\chi_\Omega(e)=\chi_0(u\rightarrow u')$.

We check that $\chi_\Omega$ is well-defined up to a null set. Assume that we have another countable directed connected graph $H'$ with a directed edge $w\rightarrow w'$ and a decoration $\chi_0'$ of $H'$, such that $(H',w,\chi_0')=(G_1,v_1,\chi_1)$ and $(H',w',\chi_0')=(G_2,v_2,\chi_2)$ in $\Omega$. By~\Cref{aeinjective}, we can assume $\chi_1,\chi_2$ are injective.
By definition, the set of the images under $\chi_0$ of edges containing $u$ in $H$ are the same to the set of the images under $\chi_1$ of edges containing $w$ in $H'$.
If $\chi_0'(w\rightarrow w')\neq \chi_0(u\rightarrow u')$, then $u'$ and $w'$ correspond to different vertices $x_1$ and $x_2$ in the neighborhood of $v$ in $G_1$. Since $\chi_1$ is injective, we have $$\bigcup_{x_1\in e\in E(G_1)}\{\chi_1(e)\}\neq \bigcup_{x_2\in e\in E(G_1)}\{\chi_1(e)\}.$$  However, by our assumption that $(H',w',\chi_0')=(H,u',\chi_0)=(G_2,v_2,\chi_2)$, we have $$\bigcup_{w'\in e\in E(H')}\{\chi'_0(e)\}=\bigcup_{u'\in e\in E(H)}\{\chi_0(e)\}=\bigcup_{v_2\in e\in E(G_2)}\{\chi_2(e)\}.$$ 
By our choice of $x_1,x_2$, this implies that $$\bigcup_{u'\in e\in E(H)}\{\chi_0(e)\}=\bigcup_{x_1\in e\in E(G_1)}\{\chi_1(e)\}~~~~\text{ and }\bigcup_{w'\in e\in E(H')}\{\chi'_0(e)\}=\bigcup_{x_2\in e\in E(G_1)}\{\chi_1(e)\}$$ Hence, we have 
$$\bigcup_{x_1\in e\in E(G_1)}\{\chi_1(e)\}=\bigcup_{v_2\in e\in E(G_2)}\{\chi_2(e)\}=\bigcup_{x_2\in e\in E(G_1)}\{\chi_1(e)\},$$
a contradiction.

Let $c_{(r)}=p_{k,r}\circ \chi_\Omega$, where $p_{k,r}:K=\prod_{k,r\geq 1}[k]^{M(k,r)}\longrightarrow [k]$ is the projection that maps an element in $K$ to the $n_r$-th coordinate of the $(k,r)$-th product.

We claim that $c_{(r)}$ is continuous under the local topology up to a null set. 
For any two edges $e_n$ from $(G_n,v_n,\chi_n)$ to $(G_n',v_n',\chi_n')$ with $n\in [2]$, such that $$d_{\Omega}((G_1,v_1,\chi_1),(G_2,v_2,\chi_2))<1/4\text{~ and ~}d_{\Omega}((G_1',v_1',\chi_1'),(G_2',v_2',\chi_2'))<1/4,$$ by~\Cref{aeinjective}, we can assume $\chi_1,\chi_2$ are injective. Thus $\{\chi_\Omega(e_n)\}$ $=\{\chi_{n}(e)\mid v_n\in e \}\cap \{\chi_{n}'(e)\mid v_n'\in e\}$ for $n= 1,2$. However, we have $\{\chi_{1}(e)\mid v_1\in e \}=\{\chi_2(e)\mid v_2\in e \}$ and $\{\chi_{1}'(e)\mid v_1'\in e \}=\{\chi_2'(e)\mid v_2'\in e \}$. Thus $\chi_\Omega(e_1)=\chi_\Omega(e_2)$.

Note that $G[c_{(r)^{-1}(F)}]_x=G[c_{j_k,r}^{-1}(F)]_x$ for all $F\subseteq [k]$ and $x\in \{(G_{j_k},v,\chi_{G_{j_k}}): v\in V(G_{j_k})\}$. Moreover, $G[c_{(r)^{-1}(F)}]_x$ is a continuous function with respect to $x$ in $\Omega$.
Thus again it follows that for all $F\subseteq [k]$:
    \begin{align*}
        \lim_{k\rightarrow\infty}\rho_{G_{j_k}}\circ c_{j_k,r}^{-1}(F)&=\lim_{k\rightarrow\infty}\sum_{v\in V(G)}\frac{1}{|V(G)|}(1-\dfrac{1}{|G[c_{j_k,r}^{-1}(F)]_x|})=\lim_{k\rightarrow\infty}\int_{\Omega} 1-\dfrac{1}{|\mathbf{G}[c_{j_k,r}^{-1}(A)]_x|}  d\mu_{G_{j_k}}\\
        &=\lim_{k\rightarrow\infty}\int_{\Omega} 1-\dfrac{1}{|\mathbf{G}[c_{(r)}^{-1}(A)]_x|}  d\mu_{G+{j_k}}=\int_{\Omega} 1-\dfrac{1}{|\mathbf{G}[c_{(r)}^{-1}(A)]_x|}  d\mu_{G}=\rho_{\mathbf{G}}\circ c_{(r)}^{-1}(F).
    \end{align*}
    
Since $ \lim_{r\rightarrow\infty}\rho_{G_i}\circ c_{i,r}^{-1}=\rho_{G_i}\circ c_i^{-1}$ and $\rho_{G_i}\circ c_i^{-1}$ converges to some $f\in\underset{i\rightarrow\infty}{\lim} \mathcal{Q}_k(G_i)$, it follows that $\rho_{n}\circ c_{(r)}^{-1}$ converges to $f$ and so $f\in \overline{\mathcal{Q}_k(\mathbf{G})}$.    
\end{proof}

Therefore, $\rho_{\mathbf{G}}$ is the quotient limit of $\{\rho_{G_n}\}_{n=1}^\infty$.
\end{proof}

\section{Concluding Remarks}\label{concluding}

In~\Cref{quotientconvergentlimit}, there exists a condition that the graph sequence is of bounded degree. We wonder if we can delete this condition, or if there exists a counterexample if the graph sequence is not of bounded degree.
\begin{ques}
    If $\{G_i\}$ is a sequence of finite graphs with rank functions $\{\rho_i\}$ quotient converging to some set function $\rho$. Does there exists a graphing $G$ such that $\rho=\rho_\mathbf{G}$ under the quotient distance?
\end{ques}

\nocite{*}
\bibliographystyle{amsplain}
\bibliography{ref}

\end{document}